\documentclass[11pt]{article}
\usepackage{amssymb, amsmath, amsthm, hyperref}
\usepackage[left=1in,top=1in,right=1in]{geometry}
\date{}
\allowdisplaybreaks

\theoremstyle{definition}
\newtheorem{theorem}{Theorem}[section]
\newtheorem{lemma}[theorem]{Lemma}
\newtheorem{proposition}[theorem]{Proposition}
\newtheorem{claim}[theorem]{Claim}
\newtheorem{fact}[theorem]{Fact}
\counterwithin{equation}{section}

\title{Evasive sets, twisted varieties, and container-clique trees}
\author{Jeck Lim\thanks{California Institute of Technology, Pasadena, CA 91125, USA. Research partially supported by an NUS Overseas Graduate Scholarship. Also supported by NSF grant DMS-1928930 while in residence at the Simons--Laufer Mathematical Sciences Institute during the Spring 2025 semester. Email: {\tt jlim@caltech.edu}.} \and Jiaxi Nie\thanks{Georgia Institute of Technology, Atlanta, GA 30332, USA. Email: {\tt jnie47@gatech.edu}.} \and Ji Zeng\thanks{Alfréd Rényi Institute of Mathematics, Budapest 1053, Hungary. Supported by ERC Advanced Grants ``GeoScape'', no. 882971 and ``ERMiD'', no. 101054936. Also supported by NSF grant DMS-1928930 while in residence at the Simons--Laufer Mathematical Sciences Institute during the Spring 2025 semester. Email: {\tt zeng.ji@renyi.hu}.}}

\date{}

\begin{document}

\maketitle
\begin{abstract}

In the affine space $\mathbb{F}_q^n$ over the finite field of order $q$, a point set $S$ is said to be $(d,k,r)$-evasive if the intersection between $S$ and any variety, of dimension $k$ and degree at most $d$, has cardinality less than $r$. As $q$ tends to infinity, the size of a $(d,k,r)$-evasive set in $\mathbb{F}_q^n$ is at most $O\left(q^{n-k}\right)$ by a simple averaging argument. We exhibit the existence of such evasive sets of sizes at least $\Omega\left(q^{n-k}\right)$ for much smaller values of $r$ than previously known constructions, and establish an enumerative upper bound $2^{O(q^{n-k})}$ for the total number of such evasive sets. The existence result is based on our study of twisted varieties. In the projective space $\mathbb{P}^n$ over an algebraically closed field, a variety $V$ is said to be $d$-twisted if the intersection between $V$ and any variety, of dimension $n - \dim(V)$ and degree at most $d$, has dimension zero. We prove an upper bound on the smallest possible degree of twisted varieties which is best possible in a mild sense. The enumeration result includes a new technique for the container method which we believe is of independent interest. To illustrate the potential of this technique, we give a simpler proof of a result by Chen--Liu--Nie--Zeng that characterizes the maximum size of a collinear-triple-free subset in a random sampling of $ \mathbb{F}_q^2$ up to polylogarithmic factors.

\end{abstract}

\section{Introduction}

Let $\mathbb{F}_q$ denote the finite field whose order is a prime power $q$. A point set $S$ in the $n$-dimensional affine space $\mathbb{F}_q^n$ is said to be \textit{$(d,k,r)$-evasive} if the intersection between $S$ and any variety in $\mathbb{F}_q^n$, of dimension $k$ and degree at most $d$, has cardinality less than $r$. In this paper, the word ``variety'' means an algebraic set that is not necessarily irreducible. For simplicity, $(1,k,r)$-evasive sets are referred to as being \textit{$(k,r)$-evasive}. The notion of evasive sets has applications in Ramsey theory~\cite{pudlak2004pseudorandom}, incidence geometry~\cite{sudakov2024evasive,milojevic2024incidence}, error-correcting codes~\cite{guruswami2011linear,dvir2012subspace}, and extremal graph theory~\cite{blagojevic2013turan,bukh2024extremal}. A similar concept called evasive subspace families is also applied to derandomization~\cite{guo2024variety}.

Since we can partition $\mathbb{F}_q^n$ into $q^{n-k}$ many $k$-dimensional affine subspaces, a $(d,k,r)$-evasive set $S$ in $\mathbb{F}_q^n$ must satisfy \begin{equation}\label{slice}
    |S| \leq (r-1) q^{n-k} / d.
\end{equation} For many fixed parameters $d$, $k$, $n$, and $r$, this inequality is not asymptotically tight as $q \to \infty$ (see e.g.~\cite{ihringer2022large}). On the other hand, the existence of $(k,r)$-evasive sets in $\mathbb{F}_q^n$ for large $r$, whose sizes asymptotically match the upper bound \eqref{slice}, are known by many authors \cite{dvir2012subspace,blagojevic2013turan,sudakov2024evasive,ihringer2022large,conlon}. In particular, Dvir and Lovett~\cite{dvir2012subspace} gave such explicit constructions for $r \geq n^k$. For $(d,k,r)$-evasive sets in general, Dvir, Koll\'ar, and Lovett~\cite{dvir2014variety} gave explicit constructions of asymptotically optimal sizes, though the requirement for $r$ is even larger. A natural question raised and discussed in \cite{sudakov2024evasive} asks for the smallest $r$ such that there exist $(d,k,r)$-evasive sets in $\mathbb{F}_q^n$ of size $\Omega\left(q^{n-k}\right)$ as $q\to \infty$. We exhibit the existence of such optimal evasive sets for much smaller values of $r$.

\begin{theorem}\label{evasive}
    For positive integers $d$ and $k \leq n$, there exists a $(d,k,r)$-evasive set in $\mathbb{F}_q^n$ of size at least $\Omega\left(q^{n-k}\right)$ with $r = c_{d,k} \cdot n^{\frac{1}{k} + \frac{1}{k-1} + \dots + \frac{1}{2} + 1}$ where $c_{d,k}$ depends only on $d$ and $k$.
\end{theorem}

We will establish the theorem above by proving the existence of twisted varieties of small degrees. For a field $\mathbb{F}$, a variety $V$ in the $n$-dimensional projective space $\mathbb{P}^n(\mathbb{F})$ is said to be \textit{$d$-twisted} if the intersection between $V$ and any variety in $\mathbb{P}^n(\mathbb{F})$, of codimension $\dim(V)$ and degree at most $d$, has dimension zero. Dvir--Lovett~\cite{dvir2012subspace} and Blagojevi\'c--Bukh--Karasev~\cite{blagojevic2013turan} introduced $1$-twisted varieties over finite fields and complex numbers respectively and independently (under different names). Dvir--Koll\'ar--Lovett~\cite{dvir2014variety} further studied the explicit construction of $d$-twisted varieties over finite fields. All these studies addressed the natural question on the smallest possible degree of twisted varieties. We give an improved upper bound on this value which is best possible in a mild sense.

\begin{theorem}\label{twisted}
    For positive integers $d$, $k \leq n$, and an algebraically closed field $\mathbb{F}$, there exists a $d$-twisted complete intersection variety in $\mathbb{P}^n(\mathbb{F})$ of dimension $n-k$ and degree at most \begin{equation*}
        c_{d,k} \cdot n^{\frac{1}{k} + \frac{1}{k-1} + \dots + \frac{1}{2} + 1}
    \end{equation*} where $c_{d,k}$ depends only on $d$ and $k$. Moreover, this bound on its degree is asymptotically tight for fixed $d,k$ as $n \to \infty$, that is, only the constant $c_{d,k}$ may be improved.
\end{theorem}

Let us remark on some limitations of our theorems. The ``asymptotic optimality'' of Theorem~\ref{twisted} is restricted to complete intersection varieties. Hence there still exists potential for non-complete-intersection twisted varieties to have degrees much smaller than the upper bound in Theorem~\ref{twisted}. It seems to be an interesting problem in algebraic geometry whether such potential is true. In most algorithmic applications, it is desirable that evasive sets are explicitly constructed, that is, efficiently and deterministically. Our Theorem~\ref{evasive} implies only an efficient probabilistic algorithm to construct evasive sets as random algebraic varieties. More precisely, the proof of Theorem~\ref{evasive} says the zero locus of randomly sampled polynomials, of degrees roughly $n^{1/i}$ for $1\leq i\leq k$, will give us a $(d,k,r)$-evasive set in $\mathbb{F}_q^n$ with high probability.

Nevertheless, random algebraic varieties can still be useful in some algorithmic scenarios since they carry more structures than naively sampled random point sets (see Theorem~10 in~\cite{guruswami2011linear}). Outside of algorithmic considerations, the existence of optimal evasive sets has many interesting consequences in incidence geometry, see e.g. Sudakov--Tomon~\cite{sudakov2024evasive} and Milojevi\'c--Sudakov--Tomon~\cite{milojevic2024incidence}. In particular, an extremal example towards a recent incidence bound in~\cite{milojevic2024incidence} is based on $(2,k,r)$-evasive sets (see Section~5.3 in~\cite{milojevic2024incidence}). Hence our Theorem~\ref{evasive} gives a lower bound for the incidence problem considered there with a better constant than that from \cite{dvir2014variety}.

\medskip

We also address the enumeration problem for evasive sets in this paper. As a consequence of Theorem~\ref{evasive}, there are $2^{\Omega\left(q^{n-k}\right)}$ many $(k,r)$-evasive sets in $\mathbb{F}_q^n$ when $r$ is moderately large in terms of $k$ and $n$. On the other hand, there are at most $2^{O\left(q^{n-k}\log q\right)}$ many point sets satisfying \eqref{slice}, and this is a trivial upper bound for the total number of such evasive sets. We show that the lower bound here is closer to the truth.
\begin{theorem}\label{counting}
    There are at most $2^{O(q^{n-k})}$ many $(k,r)$-evasive sets in $\mathbb{F}^n_q$.
\end{theorem}

This theorem is motivated by the extensive studies on the enumeration problems for $H$-free (hyper-)graphs~\cite{erdos1976asymptotic, erdos1986asymptotic,kleitman1982number,nagle2006extremal,balogh2011number,morris2016534,ferber2020supersaturated} and general position sets~\cite{BS,roche2022arcs,bhowmick2022counting,chen2023random,nenadov2024number}. In particular, Ferber, McKinley, and Samotij~\cite{ferber2020supersaturated} proved that, under some natural assumption, the number of $H$-free graphs on $n$ vertices is at most \begin{equation}\label{gwen}
    2^{O(\text{ex}(n,H))}.
\end{equation} Here, the \textit{extremal number} $\text{ex}(n,H)$ is defined as the maximum number of edges an $H$-free graph can have on $n$ vertices. While it is tempting to improve the exponent in \eqref{gwen} to $(1+o(1))\text{ex}(n,H)$, Morris and Saxton~\cite{morris2016534} showed this is impossible for $H = C_6$ using a previous result~\cite{furedi2006turan} on $\text{ex}(n,C_6)$. In fact, we cannot always expect an enumerative upper bound that is exponential in the ``extremal number'' of the considered objects. For example, an $n$-queen configuration is a placement of $n$ queens on an $n$-by-$n$ chessboard so that no two queens can attack one another. Bowtell--Keevash~\cite{bowtell2021n} and Luria--Simkin~\cite{simkin2022lower} independently proved that there are at least $((1-o(1))n/e^3)^n$ many $n$-queen configurations. Another example of this kind is the enumeration of Steiner triple systems~\cite{keevash2018counting}.

Our proof of Theorem~\ref{counting} follows the nowadays standard container method \cite{balogh2015independent,saxton2015hypergraph}, combined with a new technique we call container-clique trees. Heuristically speaking, this technique twists the common procedure of applying the hypergraph container lemma, and allows us some flexibility to assume the absence of large cliques when proving a balanced supersaturation result, which often serves as a key ingredient in the container method. We believe this technique is general enough to be applied to other enumeration problems involving the container method.

To illustrate the potential of container-clique trees, we give a simpler proof of the following result by Chen--Liu--Nie--Zeng~\cite{chen2023random}. For a real number $p\in [0,1]$, we define $\alpha(\mathbb{F}_q^2,p)$ to be the maximum size of a collinear-triple-free set that is contained in a \textit{$p$-random set} $\mathbf{S}_p \subset \mathbb{F}_q^2$. Here, by ``$p$-random'' we mean that every point of $\mathbb{F}_q^2$ is sampled into $\mathbf{S}_p$ independently with probability $p$. We characterize the behaviour of $\alpha(\mathbb{F}_q^2,p)$ up to polylogarithmic factors, denoted as the $q^{o(1)}$ terms in the statement below.

\begin{theorem}\label{no-three-in-line}
As the prime power $q\to \infty$, with high probability, we have 
\begin{equation*}
\alpha(\mathbb{F}_q^2,p)=
\left\{
    \begin{aligned}
    &\Theta(pq^{2}),~~~&q^{-2+o(1)}\leq p\leq q^{-3/2-o(1)};\\
    &q^{1/2+o(1)},~~~&q^{-3/2-o(1)}\leq p\leq q^{-1/2+o(1)};\\
    &(1 \pm o(1))pq,~~~&q^{-1/2+o(1)}\leq p\leq 1.
\end{aligned}
\right.
\end{equation*}
\end{theorem}

The old proof of this theorem in \cite{chen2023random} leverages the pseudorandomness of point-line incidence relations in $\mathbb{F}_{q}^2$ and can only conclude a weaker bound $\Theta(pq)$ in the third range. In our new proof, such an ad-hoc technique will be bypassed, allowing us to establish a stronger bound $(1 \pm o(1))pq$ in the third range.

\bigskip

The rest of this paper is organized as follows. In Section~\ref{sec:ag}, we present some preliminary results from algebraic geometry. In Section~\ref{sec:boris}, we prove Theorems~\ref{evasive} and~\ref{twisted}. In Section~\ref{sec:cctree}, we introduce the container-clique tree and use it to prove Theorem~\ref{no-three-in-line}. Section~\ref{sec:counting} is devoted to the proof of Theorem~\ref{counting}. Section~\ref{sec:remark} includes some final remarks. Throughout our paper, the asymptotic notations are understood with the prime power $q$ tending to infinity unless otherwise stated; all logarithms are in base two; we omit floors and ceilings since they are not crucial for the sake of clarity in our presentation.

\section{Some results from algebraic geometry}\label{sec:ag}

Other than the terminological difference that a variety can be reducible for us, the readers are referred to \cite{hartshorne} for basic notions and facts of algebraic geometry. The results stated in this section are over an algebraically closed field $\mathbb{F}$. We write $\mathbb{P}^n(\mathbb{F})$ as $\mathbb{P}^n$ for abbrevity. A \textit{$k$-plane} refers to a $k$-dimensional projective subspace inside another projective space of higher dimension.

We use $P^n_{d}$ to denote the space of all homogeneous polynomials in $n+1$ variables of degree $d$ up to scaling over the field $\mathbb{F}$. We consider $P^n_{d}$ as a variety by identifying it with the projective space $\mathbb{P}^{\binom{d+n}{n}-1}$. A variety $V \subset \mathbb{P}^n$ is a \textit{complete intersection} if there exist $n - \dim V$ homogeneous polynomials that generate all homogeneous polynomials vanishing on $V$. We need the fact below to argue that most polynomial sequences define complete intersections.
\begin{fact}\label{reduced}
    There is a dense subset of points $(f_1,\ldots,f_k) \in P^{n}_{d_1}\times\dots\times P^n_{d_k}$ such that the subscheme defined by $f_1,\ldots,f_k$ in $\mathbb{P}^n$ is reduced.
\end{fact}

To see this, we can argue that there is a dense subset of points $(f_1,\ldots,f_k)$ such that the subscheme defined by $f_1,\ldots,f_k$ satisfies the Jacobian criterion for smoothness everywhere (see Exercise I.5.15 in \cite{hartshorne}), then we can use the fact that a smooth projective scheme is necessarily reduced (see Example 10.0.3 in \cite{hartshorne}). 

For a variety $V \subset \mathbb{P}^n$, its \textit{Hilbert function} can be defined as \begin{equation}\label{hilbert1}
    \varphi_V(d) = \dim P^n_{d} - \dim I_d(V),
\end{equation} where $I_d(V) = \{ f \in P^n_{d}:~ \text{$f$ vanishes on $V$}\}$. Here, $I_d(V)$ is regarded as a subvariety of $P^n_{d}$. A key ingredient in our proofs of Theorems~\ref{evasive} and~\ref{twisted} is a basic estimate as follows (Lemma~6 in~\cite{bukh2024extremal}, see also \cite{sombra1997bounds}).
\begin{lemma}\label{hilbert2}
    If a variety $V \subset \mathbb{P}^n$ has dimension $k$, then $\varphi_V(d) \geq \binom{d+k}{k}$.
\end{lemma}

The ``moreover'' part of our Theorem~\ref{twisted} relies on the fact below (Theorem~2.1(b) in \cite{debarre1998vari}).
\begin{theorem}\label{fano}
    A variety $V \subset \mathbb{P}^n$ defined by polynomials of degree $d_1,d_2,\dots,d_s$ must contain a $k$-plane if $n \geq 2k + s$ and \begin{equation*}
        (k+1)(n-k) \geq \sum_{i=1}^s \binom{d_i+k}{k}.
    \end{equation*}
\end{theorem}

We shall use Chow varieties to enumerate subvarieties in a projective space of given dimension and degree. An \textit{effective algebraic $k$-cycle} $X = \sum_i m_iV_i$ is a formal sum of $k$-dimensional irreducible subvarieties $V_i$ in $\mathbb{P}^n$ with $m_i$ being positive integers. The \textit{degree} $\deg(X)$ is defined as $\sum_i m_i\deg(V_i)$, and the \textit{support} $|X|$ is defined as the variety $\cup_i V_i$. The \textit{Chow variety} $\text{Ch}(d,k,n)$ is a projective variety whose points are in one-to-one correspondence with effective algebraic $k$-cycles of degree $d$ in $\mathbb{P}^n$. We refer our readers to \cite{gelfand1994discriminants} for an introduction to its theory (and Section~I.9 in \cite{samuel1967} for a treatment over fields of arbitrary characteristic). The dimensions of Chow varieties are determined by Azcue~\cite{azcue1992dimension} and Lehmann~\cite{lehmann2017asymptotic} independently.
\begin{theorem}\label{chow}
    $\dim \text{Ch}(d,k,n) =\max \left\{ d(k+1)(n-k), \binom{d+k+1}{k+1}-1 + (k+2)(n-k-1) \right\}$.
\end{theorem}

We will need the following statement for a minor technicality.
\begin{proposition}\label{fineorcoarse}
    The set of pairs $(p,X) \in \mathbb{P}^n \times \text{Ch}(d,k,n)$ satisfying $p \in |X|$ is Zariski closed in the variety $\mathbb{P}^n \times \text{Ch}(d,k,n)$.
\end{proposition}

Due to a lack of reference, we give a brief explanation of this fact, and we shall follow the treatment in \cite{gelfand1994discriminants}. The \textit{Grassmannian} $\text{Gr}(k,n)$ is an irreducible projective variety whose points are in one-to-one correspondence with $k$-dimensional linear subspaces in $\mathbb{F}^n$. In particular, $\text{Gr}(k+1,n+1)$ parametrizes all $k$-planes in $\mathbb{P}^n$. The Pl\"ucker map embeds $\text{Gr}(k,n)$ into a larger projective space $\mathbb{P}^m$ (see \cite{griffiths2014principles}). Given a $k$-dimensional irreducible subvariety $V \subset \mathbb{P}^n$ of degree $d$, the set of $(n-k-1)$-planes in $\mathbb{P}^n$ that has non-empty intersection with $V$ is a subvariety of $\text{Gr}(n-k,n+1)$, which is also the zero locus of some element $R_V \in S(d,k,n)$. Here, $S(d,k,n)$ is a vector space consisting of a zero element and all restrictions of degree-$d$ homogeneous polynomials from $\mathbb{P}^m$ onto $\text{Gr}(n-k,n+1)$ under the Pl\"ucker embedding. For a $k$-cycle $X = \sum_i m_iV_i$ of degree $d$, we define its \textit{Chow form} as $R_X = \prod R_{V_i}^{m_i}$, which is unique up to scaling. Hence, we can regard $R_X$ as an element in the projectivization $\mathbb{P}(S(d,k,n))$, and the map $X \mapsto R_X$ is proven (in \cite{chow1937algebraischen}) to be a closed embedding of variety $\text{Ch}(d,k,n) \to \mathbb{P}(S(d,k,n))$.

To prove Proposition~\ref{fineorcoarse}, we first notice that the set of pairs $(f,F)$ satisfying $f(F) = 0$ is closed in $\mathbb{P}(S(d,k,n)) \times \text{Gr}(n-k,n+1)$. Together with the fact that $\text{Ch}(d,k,n)$ is closed in $\mathbb{P}(S(d,k,n))$, this implies the set of pairs $(X,F)$ satisfying $R_X(F) = 0$ is closed in $\text{Ch}(d,k,n) \times \text{Gr}(n-k,n+1)$. By the definition of Chow forms, this means the set of pairs $(X,F)$ satisfying $|X| \cap F \neq \emptyset$ is closed in $\text{Ch}(d,k,n) \times \text{Gr}(n-k,n+1)$. On the other hand, the set of pairs $(p,F)$ satisfying $p \in F$ is known to be closed in $\mathbb{P}^n \times \text{Gr}(n-k,n+1)$, see Section~I.5 ``Universal Bundles'' in \cite{griffiths2014principles}. Hence, the set $\mathcal{T}$ of triples $(p,X,F)$ satisfying $p\in F$ and $|X| \cap F \neq \emptyset$ is closed in $\mathbb{P}^n \times \text{Ch}(d,k,n) \times \text{Gr}(n-k,n+1)$. We also notice that the collection of $(n-k-1)$-planes containing any fixed point in $\mathbb{P}^n$ forms an irreducible subvariety in $\text{Gr}(n-k,n+1)$ that is isomorphic to $\text{Gr}(n-k-1,n)$. Now, we consider the canonical projection $\pi: \mathcal{T} \to \mathbb{P}^n \times \text{Ch}(d,k,n)$. It is easy to argue that the pairs $(p,X)$ satisfying $p \in |X|$ are exactly those satisfying $\dim\left( \pi^{-1}(p,X) \right) \geq \dim \text{Gr}(n-k-1,n)$. Then the well-known fact below asserts that a set written in this form is closed (see Exercise~II.3.22 in \cite{hartshorne}). 

\begin{lemma}\label{fiberdimension}
    Let $f:X\to Y$ be a surjective morphism of projective varieties. For any integer $h$, let $C_h$ be the set of points $y\in Y$ such that the fibre $f^{-1}(y)$ has dimension at least $h$, then $C_h$ is closed. Moreover, if both $X$ and $Y$ are irreducible, $C_e = Y$ for $e = \dim X-\dim Y$.
\end{lemma}


\section{Proofs of Theorems~\ref{evasive} and~\ref{twisted}}\label{sec:boris}

Our proofs are inspired by a random-algebraic construction of Boris Bukh \cite{bukh2024extremal}. However, we employ a dimension-counting argument that avoids probabilistic notions and works over algebraically closed fields (an assumption for many results from algebraic geometry).

\begin{proof}[Proof of Theorem~\ref{twisted}]
    For positive integers $d_1,\dots,d_k$ to be determined later, we define \begin{equation*}
        \mathcal{B} = \{ (f_1,\ldots,f_k,X)\in P^{n}_{d_1}\times\dots\times P^n_{d_k}\times \text{Ch}(d,k,n) :~ \dim ( Z(f_1,\dots,f_k) \cap |X| )>0 \}.
    \end{equation*} We check that $\mathcal{B}$ is an algebraic variety. Indeed, we consider the auxiliary set \begin{equation*}
        \mathcal{B}' = \{ (f_1,\ldots,f_k,p,X)\in P^{n}_{d_1}\times\dots\times P^n_{d_k} \times \mathbb{P}^n \times \text{Ch}(d,k,n) :~ f_1(p)=\dots=f_k(p)=0,~ p \in |X| \}.
    \end{equation*} It follows from Proposition~\ref{fineorcoarse} that $\mathcal{B}'$ is a variety. Using the canonical projection $\pi: \mathcal{B}' \to P^{n}_{d_1}\times\dots\times P^n_{d_k}\times \text{Ch}(d,k,n)$, we see $\mathcal{B}$ is exactly the points in $P^{n}_{d_1}\times\dots\times P^n_{d_k}\times \text{Ch}(d,k,n)$ whose fibres under $\pi$ have positive dimension, so $\mathcal{B}$ is a variety by Lemma~\ref{fiberdimension}.

    Next, we consider canonical projections $\pi_1: \mathcal{B} \to P^{n}_{d_1}\times\dots\times P^n_{d_k}$ and $\pi_2: \mathcal{B} \to \text{Ch}(d,k,n)$. For an arbitrary
    $X \in \text{Ch}(d,k,n)$, we shall prove the following bound on its fibre dimension \begin{equation}\label{fiber}
        \dim \left(\pi_2^{-1}(X)\right) \leq \dim\left( P^{n}_{d_1}\times\dots\times P^n_{d_k} \right) - \min_{1\leq i\leq k} \binom{d_i+k+1-i}{k+1-i}.
    \end{equation} Notice that Theorem~\ref{chow} says $\dim \text{Ch}(d,k,n) = O(n)$ as $n\to \infty$. By taking $d_i = \Omega\left(n^{1/(k+1-i)}\right)$, we can guarantee \begin{equation}\label{choice}
        \min_{1\leq i\leq k} \binom{d_i+k+1-i}{k+1-i} > \dim \text{Ch}(d,k,n).
    \end{equation} Let $\mathcal{B}_0$ be an irreducible component of $\mathcal{B}$ with maximum dimension. It is elementary that $\pi_2(\mathcal{B}_0)$ is also irreducible in $\text{Ch}(d,k,n)$. We apply the ``moreover'' part of Lemma~\ref{fiberdimension} to the restriction of $\pi_2$ on $\mathcal{B}_0$, Together with \eqref{fiber} and \eqref{choice}, we have, for any $X\in \pi_2(\mathcal{B}_0)$, \begin{align*}
        \dim(\mathcal{B}) = \dim(\mathcal{B}_0) &\leq \dim (\pi_2^{-1}(X))+\dim (\pi_2(\mathcal{B}_0))\\
                        &\leq \dim\left( P^{n}_{d_1}\times\dots\times P^n_{d_k} \right) - \min_{1\leq i\leq k} \binom{d_i+k+1-i}{k+1-i} + \dim (\text{Ch}(d,k,n))\\
                        &< \dim\left( P^{n}_{d_1}\times\dots\times P^n_{d_k} \right).
    \end{align*} Hence, we have $\dim\left(\pi_1(\mathcal{B})\right) < \dim ( P^{n}_{d_1}\times\dots\times P^n_{d_k} )$. Together with Fact~\ref{reduced}, there must exist polynomials $f_i$ of degrees $d_i$ such that $\dim ( Z(f_1,\dots,f_k) \cap |X| ) = 0$ for all $X \in \text{Ch}(d,k,n)$ and the subscheme $S$ defined by $f_1,\dots,f_k$ in $\mathbb{P}^n$ is reduced. By the former condition, we must have $\dim S = n - k$. By the latter condition, $S$ is also an abstract variety, hence $Z(f_1,\dots,f_k) = S$ is a complete intersection variety.
    
    Since any variety of dimension $k$ and degree at most $d$ is contained in the support $|X|$ of some cycle $X \in \text{Ch}(d,k,n)$, the variety $Z(f_1,\dots,f_k)$ must be $d$-twisted. Furthermore, it has degree \begin{equation*}
        d_1 \cdot d_2 \cdots \cdot d_k = O\left(n^{\frac{1}{k} + \frac{1}{k-1} + \dots + \frac{1}{2} + 1}\right),
    \end{equation*} where the implicit constant depends only on $d$ and $k$.

    For the ``moreover'' part, suppose $V = Z(f_1,\dots,f_k)$ is a complete intersection variety that is $d$-twisted. In particular, $V$ is 1-twisted. Let $d_i$ denote the degree of $f_i$ and assume without loss of generality that $d_1 \leq d_2 \leq \dots \leq d_k$. Notice that $Z(f_1,\dots,f_i)$ cannot contain a $(k+1-i)$-plane $F$ for all $1\leq i\leq k$. Indeed, otherwise, let $H$ be any hyperplane containing $F$, then $H\cap V$ contains the algebraic set $F \cap Z(f_{i+1},\dots,f_k)$, which has a positive dimension, contradicting the fact that $V$ is 1-twisted. By Theorem~\ref{fano} applied to the variety $Z(f_1,\dots,f_i)$, we have for sufficiently large $n$, \begin{equation*}
        (k+2-i)(n-k-1+i) <\sum_{j=1}^i \binom{d_j+k+1-i}{k+1-i}.
    \end{equation*} In particular, $d_i=\max\{d_1,\dots,d_i\} = \Omega\left(n^{1/(k+1-i)}\right)$. Since the degree $\deg V$ equals the product $\prod_i d_i$ when $V$ is a complete intersection, we can conclude the proof.

    It suffices for us to prove \eqref{fiber}, which follows by induction on $\ell$ in the following statement.
    \begin{claim}
    For positive integers $\ell \leq k$, $d_1,\dots,d_\ell$, and $X \in \text{Ch}(d,k,n)$, the set $\mathcal{B}_X = \{ (f_1,\dots,f_\ell)\in P^{n}_{d_1}\times\dots\times P^n_{d_\ell} :~ \dim ( Z(f_1,\dots,f_\ell) \cap |X| )> k -\ell \}$ is a subvariety in $P^{n}_{d_1}\times\dots\times P^n_{d_\ell}$ of codimension at least \begin{equation*}
        \min_{1\leq i\leq \ell} \binom{d_i+k+1-i}{k+1-i}.
    \end{equation*}
    \end{claim}
    
    We skip the argument that $\mathcal{B}_X$ is a variety since it is very similar to that for $\mathcal{B}$. We can assume $|X|$ is irreducible without loss of generality. Indeed, suppose the irreducible components of $|X|$ are $V_1,\dots, V_s$, if $\dim( Z(f_1,\ldots,f_\ell) \cap |X|)> k-\ell$, then $\dim( Z(f_1,\ldots,f_\ell) \cap V_i)>k-\ell$ for some $i$. This means $\mathcal{B}_X \subset \mathcal{B}_{V_1}\cup\dots\cup \mathcal{B}_{V_s}$, thus $\dim (\mathcal{B}_X)  \leq \max_i \dim \left(\mathcal{B}_{V_i}\right)$. Here, $V_i$ can be regarded as an effective algebraic cycle itself. For the base case when $\ell = 1$, $\mathcal{B}_X$ consists of all polynomials that vanish on $|X|$ (as it is irreducible) and the claim follows from \eqref{hilbert1} and Lemma~\ref{hilbert2}.
        
    For the inductive step, we consider the partition $\mathcal{B}_X = \mathcal{B}_1 \cup \mathcal{B}_2$ where $(f_1,\ldots,f_\ell)$ is in $\mathcal{B}_1$ if $\dim ( Z(f_1) \cap |X| ) = k$, and in $\mathcal{B}_2$ if $\dim ( Z(f_1) \cap |X| ) = k - 1$. Again, since $|X|$ is irreducible, $\dim ( Z(f_1) \cap |X| ) = k$ if and only if $f_1$ vanishes on $|X|$, and by \eqref{hilbert1} and Lemma~\ref{hilbert2}, these polynomials form a subvariety of codimension at least $\binom{d_1 + k}{k}$ in $P^n_{d_1}$. For any fixed $f_1$ with $\dim ( Z(f_1) \cap |X| ) = k - 1$, we can apply inductive hypothesis to the variety $Z(f_1) \cap |X|$ and deduce that the set $\{ (f_2,\dots,f_\ell):~ \dim ( Z(f_1,f_2,\dots,f_\ell) \cap |X| )> k -\ell \}$ is a subvariety in $P^{n}_{d_2}\times\dots\times P^n_{d_\ell}$ of codimension at least \begin{equation*}
            \min_{2\leq i\leq \ell} \binom{d_i+k+1-i}{k+1-i}.
    \end{equation*} Using these bounds on the dimensions of $\mathcal{B}_1$ and $\mathcal{B}_2$, we can upper bound $\dim \mathcal{B}_V$ and conclude the inductive step of our claim. 
\end{proof}

\begin{proof}[Proof of Theorem~\ref{evasive}]
    Let $\overline{\mathbb{F}}_q$ be the algebraic closure of $\mathbb{F}_q$. A tuple of polynomials $(f_1,\ldots,f_k)\in P^n_{d_1}\times\dots\times P^n_{d_k}$ over $\overline{\mathbb{F}}_q$ is said to be \textit{good} if $Z(f_1,\dots,f_k)$ is a $d$-twisted complete intersection variety in $\mathbb{P}^n(\overline{\mathbb{F}}_q)$. According to the proof of Theorem~\ref{twisted}, bad tuples form a subvariety of $P^n_{d_1}\times\dots\times P^n_{d_k}$ of smaller dimension for certain choices of $d_i$ such that \begin{equation*}
        d_1 \cdot d_2 \cdots \cdot d_k = c_{d,k} \cdot n^{\frac{1}{k} + \frac{1}{k-1} + \dots + \frac{1}{2} + 1}.
    \end{equation*} Hence, a simple application of the Schwartz--Zippel lemma \cite{DeMilloLipton1978,Zippel1979,Schwartz1980} shows that the number of bad tuples defined over $\mathbb{F}_q$ is $O(q^{N-1})$ as $q \to \infty$ where $N = \dim(P^n_{d_1}\times\dots\times P^n_{d_k})$. Since there are $\Omega(q^N)$ many tuples defined over $\mathbb{F}_q$, there exists a good tuple $(f_1,\ldots,f_k)$ defined over $\mathbb{F}_q$ for sufficiently large $q$. We let $V = Z(f_1,\dots,f_k)$ and denote its $\mathbb{F}_q$-points as $V(\mathbb{F}_q)$. Due to twisted-ness and Bézout's theorem, the intersection between $V$ and any variety, of dimension $k$ and degree at most $d$, has cardinality at most $r$, where $r = d\cdot\prod_id_i$. According to the Lang--Weil bound~\cite{lang1954number}, we have $|V(\mathbb{F}_q)| = (1 \pm o(1))q^{n-k}$. By covering $\mathbb{P}^n$ with $n+1$ affine spaces isomorphic to $\mathbb{F}_q^n$, the intersection of $V(\mathbb{F}_q)$ with some member in this affine cover gives us a $(d,k,r)$-evasive set as claimed.
\end{proof}

\section{Container-clique trees and Theorem~\ref{no-three-in-line}}\label{sec:cctree}

A \textit{container-clique tree} for a hypergraph $\mathcal{H}$ is a rooted tree $T$ such that \begin{enumerate}
    \item Every node $x$ of $T$ is labelled with a sequence $C^x_0, C^x_1, \dots, C^x_{\ell_x}$ of subsets of $V(\mathcal{H})$.
    \item Every vertex subset $C^x_i$ with $i \neq 0$ is a clique of $\mathcal{H}$.
    \item Every independent set $I$ of $\mathcal{H}$ is \textit{contained} in a leaf $x$ of $T$, that is \begin{equation*}
        I \subset C^x_0 \cup C^x_1 \cup \dots \cup C^x_{\ell_x}.
    \end{equation*}
\end{enumerate}

The common approach of the container method is an iterative application of the container lemma to construct a suitable tree such that each node is labelled by one vertex subset called a container. With the notion of container-clique trees, we can choose to apply the container lemma when there are few large cliques in the hypergraph considered while building the tree, or to collect many large cliques into node labels efficiently. The efficiency of the latter case stems from the fact that cliques and independent sets have small intersections. To demonstrate this idea, we give a new proof of Theorem~\ref{no-three-in-line} which is simpler than that in \cite{chen2023random}. The essential part of Theorem~\ref{no-three-in-line} is the following statement. 
\begin{proposition}\label{no-three-in-line_key}
For $p \geq \log^3 q /\sqrt{q}$, we have $\alpha(\mathbb{F}^2_q,p) \leq (1+o(1))pq$ with high probability.
\end{proposition}
\begin{proof}[Proof Sketch of Theorem~\ref{no-three-in-line}]
    In the first range, the upper bound for $\alpha(\mathbb{F}^2_q,p)$ comes from the expected size of a $p$-random set in $\mathbb{F}_q^2$. The lower bound is due to the fact that deleting a point from every collinear triple will not significantly reduce the size of a $p$-random set. In the third range, the upper bound for $\alpha(\mathbb{F}^2_q,p)$ is due to Proposition~\ref{no-three-in-line_key}. The lower bound is by considering the intersection of a $p$-random set with the moment curve $\{(x,x^2):~ x\in \mathbb{F}_q\}$, which is known to be collinear-triple-free. In the second range, we make use of the non-decreasing property of $\alpha(\mathbb{F}^2_q,p)$ as a function of $p$. The claimed bounds follow from those in the other two ranges. For more details, please consult the proof of Theorem~1.1 in \cite{chen2023random}.
\end{proof}

We now state the hypergraph container lemma in one of its forms. This well-known tool was initially developed by Balogh--Morris--Samotij~\cite{balogh2015independent} and Saxton--Thomason~\cite{saxton2015hypergraph} independently. The readers unfamiliar with basic notions of hypergraphs are referred to \cite{balogh2015independent,saxton2015hypergraph} as well. For an $r$-uniform hypergraph $\mathcal{H}$ whose vertex set is denoted as $V(\mathcal{H})$ and edge set is denoted as $E(\mathcal{H})$, we use $\Delta_i(\mathcal{H})$ to denote the maximum degree of any $i$-subset of $V(\mathcal{H})$, that is, the maximum number of members in $E(\mathcal{H})$ containing any given subset of $V(\mathcal{H})$ of cardinality $i$. An \textit{independent set} is a vertex subset of which every $r$-set is not an edge, and a \textit{clique} is a vertex subset of which every $r$-set is an edge. For $S \subset V(\mathcal{H})$, the \textit{induced sub-hypergraph} $\mathcal{H}[S]$ is the hypergraph whose vertex set is $S$ and edge set consists of all edges of $\mathcal{H}$ contained in $S$. The following lemma is Theorem 4.2 in~\cite{morris2016534}.

\begin{lemma}\label{HCL}
Let $r \ge 2$ be an integer and $c>0$ be sufficiently small with respect to $r$. If $\mathcal{H}=(V,E)$ is an $r$-uniform hypergraph and $0<\tau<1/2$ is a real number such that \begin{equation*}
        \Delta_i(\mathcal{H}) \leq c \cdot \tau^{i-1} \frac{|E|}{|V|} \quad\text{for all $2\leq i\leq r$,}
\end{equation*} then there exists a family $\mathfrak{C}$ of vertex subsets of $\mathcal{H}$ with the following properties:
\begin{itemize}
            \item[(a)] Every independent set of $\mathcal{H}$ is contained in some $C \in \mathfrak{C}$.
            \item[(b)] $\log |\mathfrak{C}| \leq c^{-1} \cdot \tau |V| \cdot \log(1/\tau)$.
            \item[(c)] For every $C \in \mathfrak{C}$, we have $|E(\mathcal{H}[C])| \leq (1 - c)|E|$.
\end{itemize}
\end{lemma}

The proof of Proposition~\ref{no-three-in-line_key} relies on the following simple fact about collinear triples.
\begin{lemma}\label{no-three-in-line_supersaturation}
    For a point set $P$ in $\mathbb{F}_q^2$, the number of collinear triples in $P$ is at least \begin{equation*}
        \frac{1}{3}\cdot|P|(q+1)\cdot\binom{\frac{|P|-1}{q+1}}{2}.
    \end{equation*}
\end{lemma}
\begin{proof}
    For each $u \in P$, there are $q+1$ lines in $\mathbb{F}_q^2$ containing $u$. Denote the intersection between $P\setminus u$ and each of these lines by $L^u_i$ for $1\leq i\leq q+1$. For each pair $\{v, w\}$ in $L^u_i$, we count the collinear triple $\{u,v,w\}$. By Jensen's inequality, we estimate the count of collinear triples in this process \begin{equation*}
        \sum_{u\in P}\sum_{i=1}^{q+1} \binom{|L^u_i|}{2} \geq \sum_u (q+1) \cdot \binom{\frac{\sum_i |L^u_i|}{q+1}}{2} = |P| (q+1)\cdot\binom{\frac{|P|-1}{q+1}}{2}.
    \end{equation*} Each triple is counted at most three times, so we conclude the lemma.
\end{proof}

\begin{proof}[Proof of Proposition~\ref{no-three-in-line_key}]
    We fix an arbitrary $\epsilon > 0$ for the $o(1)$-term. Let $\mathcal{H}$ be the hypergraph with $V(\mathcal{H}) = \mathbb{F}^2_q$ and $E(\mathcal{H})$ consisting of all collinear triples. We build a container-clique tree $T$ for $\mathcal{H}$ through the following process. We start with $T$ having one root node labelled by $\mathbb{F}_q^2$; iteratively, if there is a node $x \in T$ with $|C^x_0| \geq (1+\epsilon)q$, we perform an operation on $x$ as described below; the process stops when no such node exists. We remark that a subset $P$ in $\mathbb{F}_q^2$ with $|P| \geq (1+\epsilon)q$ contains at least $\Omega(|P|^3/q)$ collinear triples by Lemma~\ref{no-three-in-line_supersaturation} when $q$ is sufficiently large, which is a condition we can assume as $q \to \infty$.

    Now, we describe the operation on a node $x$. We take $C = C^x_0$ first. If there exists a \textit{rich} line $L$ in $\mathbb{F}_q^2$, that means, $|C \cap L| \geq |C^x_0|/\sqrt{q}$, then we delete the subset $C \cap L$ from $C$. We repeat this deletion step until either (i) $|C| < \max\{|C^x_0|/2,(1+\epsilon)q\}$ or (ii) no rich line exists. We use $K_1,K_2,\dots,K_\ell$ to denote the deleted subsets. Note that every $K_i$ is a clique of $\mathcal{H}$. If the deletion steps ended with (i), then we create a child $y$ for $x$ with \begin{equation*}
        C^y_0 = C,~ C^y_1 = C^x_1, \dots, C^y_{\ell_x}=C^x_{\ell_x},~ C^y_{\ell_x+1} = K_1,\dots, C^y_{\ell_x+\ell} = K_\ell.
    \end{equation*} In other words, $y$ copies the label of $x$, but with $C^x_0$ replaced by $C$ and all $K_i$'s appended.
    
    If the deletion steps ended with (ii), we consider an auxiliary hypergraph $\mathcal{H}' = \mathcal{H}[C]$, that is, the sub-hypergraph induced on $C$. Notice that we have $|C| \geq \max\{|C^x_0|/2,(1+\epsilon)q\}$ in this case. Hence we can apply Lemma~\ref{no-three-in-line_supersaturation} to $C$ and conclude that $E(\mathcal{H}') \geq \Omega(|C|^3/q)$. On the other hand, we can upper bound \begin{equation*}
        \Delta_2(\mathcal{H}') \leq 2|C|/\sqrt{q} \qquad\text{and}\qquad \Delta_3(\mathcal{H}') \leq 1.
    \end{equation*} The first inequality above, crucially, follows from the absence of rich lines, and the second inequality is trivial. Thereafter, we apply Lemma~\ref{HCL} to $\mathcal{H}'$ with $\tau = c' \cdot \sqrt{q}/|C|$ to obtain a family $\mathfrak{C}$ of subsets of $C$. Here, $c'$ is a sufficiently large constant such that Lemma~\ref{HCL} can be applied with a sufficiently small $c$ as required there. For each $C' \in \mathfrak{C}$, we create a child $y$ for $x$ with $C^y_0 = C'$, $C^y_1 = C^x_1, \dots, C^y_{\ell_x}=C^x_{\ell_x}$, $C^y_{\ell_x+1} = K_1,\dots, C^y_{\ell_x+\ell} = K_\ell$. Thus, the node operation is described.

    We argue that every independent set of $\mathcal{H}$ is contained in a leaf of $T$. Since the root of $T$ contains all independent sets, it suffices for us to prove that if an independent set $I$ is contained in a non-leaf $x$, then one of its children $y$ must contain $I$. Indeed, this is easy to prove if the deletion steps of $x$ ended with (i). If the deletion steps of $x$ ended with (ii), then $I\cap C^x_0$ can be decomposed into $I\cap C$ and $I\cap K_i$ for $1\leq i \leq \ell$. Being an independent set of $\mathcal{H}'$, $I\cap C$ is contained in $C^y_0$ for some child $y$ of $x$ by Lemma~\ref{HCL}(a). Every set such as $I\cap K_i$ or $I\cap C^x_i$ with $i\neq 0$ are contained in some $C^y_j$. This means $y$ contains $I$. It is guaranteed throughout our process that $C^x_i$ is a clique for $x \in T$ and $i\neq 0$. Therefore, $T$ is a container-clique tree for $\mathcal{H}$.
    
    Next, we examine some quantitative properties of $T$. Firstly, each node operation either shrinks $|C^x_0|$ by a constant factor, or reduces $|C^x_0|$ to below $(1+\epsilon)q$, or shrinks $|E(\mathcal{H}[C^x_0])|$ by a constant factor according to Lemma~\ref{HCL}(c), so the height of $T$ is at most $O(\log q)$. Secondly, since each node operation extends the label by at most $\sqrt{q}$ additional cliques, every node has label length at most $O(\sqrt{q}\log q)$. By our choice of $\tau$ and Lemma~\ref{HCL}(b), the degree of a node (that is, the number of its children) is at most \begin{equation*}
         |\mathfrak{C}| \leq 2^{O\left(\tau |V| \cdot \log(1/\tau)\right)} = 2^{O\left(\sqrt{q}\log q\right)}.
    \end{equation*} Hence, $T$ has at most $2^{O\left(\sqrt{q}\log^2 q\right)}$ many leaves. Lastly, we have $|C^x_0| < (1+\epsilon)q$ for all leaf $x$.

    Finally, we estimate the probability that a $p$-random set $\mathbf{S}_p \subset \mathbb{F}_q^2$ contains a collinear-triple-free set of size $m$, which will also be an independent set in $\mathcal{H}$. Inside $T$, we denote the number of leaves as $\nu$, the maximum size of $C^x_0$ for leaf $x$ as $\chi$, the maximum size of $C^x_i$ for $i\neq 0$ as $\kappa$, and the maximum label length $\ell_x$ as $\lambda$. If $\mathbf{S}_p$ contains an independent set $I$, since $I$ is contained in some leaf $x \in T$ and the intersection between an independent set and a clique is less than the uniformity $r = 3$, we must have $I \subset C^x_0\cup Y_1\cup \dots\cup Y_{\ell_x}$ for some $Y_i \subset C^x_i$  with $|Y_i| = 3$ for $1\leq i\leq \ell_x$. The number of choices of such $(x,Y_1,\dots,Y_{\ell_x})$ is at most $\nu\binom{\kappa}{r}^{\lambda}\leq 2^{O\left(\sqrt{q}\log^2 q\right)}$.
    
    For a fixed choice $(x,Y_1,\dots,Y_{\ell_x})$, we can estimate the expectation below \begin{equation*}
        \mathbb{E}(|\mathbf{S}_p \cap C^x_0\cup Y_1\cup \dots\cup Y_{\ell_x}|) \leq p(\chi+3\lambda)\le p\cdot \left((1+\epsilon)q+O(\sqrt{q}\log q)\right)\le (1+2\epsilon)pq.
    \end{equation*} Then we can apply Chernoff's bound (see Lemma~5 in \cite{balogh2025maximum}) with $m = (1+3\epsilon)pq$ to show \begin{equation*}
        \Pr[|\mathbf{S}_p \cap C^x_0\cup Y_1\cup \dots\cup Y_{\ell_x}| \geq m] \leq 2^{-cpq},
    \end{equation*} where the constant $c > 0$ depends only on $\epsilon$. Hence, we can apply the union bound with $p \geq \log^3 q /\sqrt{q}$ to conclude \begin{align*}
        \Pr[\alpha(\mathbb{F}^2_q,p)\geq m] &\leq \sum_{x,Y_1,\dots,Y_{\ell_x}}\Pr[|\mathbf{S}_p \cap C^x_0\cup Y_1\cup \dots\cup Y_{\ell_x}| \geq m]\\
        &\leq 2^{O\left(\sqrt{q}\log^2 q\right)}/2^{cpq}\\
        &= o(1).
    \end{align*} Therefore, we conclude the proof since $\epsilon$ is chosen arbitrarily.
\end{proof}

\section{Proof of Theorem~\ref{counting}}\label{sec:counting}

In this section, a \textit{$k$-flat} refers to a $k$-dimensional affine subspace inside another affine space of higher dimension. A set of $r$ points that is contained in a $k$-flat will be called a \textit{$(k,r)$-set}. Besides the idea of container-clique trees, another ingredient of our Theorem~\ref{counting} is a supersaturation lemma inspired by the recent paper of Balogh and Luo~\cite{balogh2025maximum}.

\begin{lemma}\label{counting_supersaturation}
    For positive integers $k < n, r$ and real number $c > 0$, there exists a sufficiently large $\theta$ such that the following holds. If $P \subset \mathbb{F}_q^{n}$ satisfies $ |P| \geq \theta q^{n-k}$ and no $k$-flat intersects $P$ in more than $2|P|/\sqrt{q}$ points, then there exists a hypergraph $\mathcal{H}$ with $V(\mathcal{H}) = P$ and $E(\mathcal{H})$ consisting of some $(k,r)$-sets such that, for all $2\leq i\leq r$, \begin{equation*}
        \Delta_i(\mathcal{H}) \leq c \cdot \tau^{i-1} \frac{|E(\mathcal{H})|}{|V(\mathcal{H})|} \quad\text{with}\quad \tau = \frac{\theta q^{n-k}}{|P|q^\epsilon} ~\text{and}~ \epsilon = \frac{1}{2r},
    \end{equation*} and $\Delta_1(\mathcal{H})\leq \theta |E(\mathcal{H})|/|V(\mathcal{H})|$.
\end{lemma}

\begin{proof}[Proof of Theorem~\ref{counting}]

    We assume $n,r > k$ since the other cases are trivial. Let $\mathcal{H}$ be the hypergraph with $V(\mathcal{H}) = \mathbb{F}^n_q$ and $E(\mathcal{H})$ consisting of all $(k,r)$-sets. We build a container-clique tree $T$ for $\mathcal{H}$ through the following process. We start with $T$ having one root node labelled by $\mathbb{F}_q^n$; iteratively, if there is a node $x \in T$ with $|C^x_0| \geq 2\theta q^{n-k}$ for the $\theta$ in Lemma~\ref{counting_supersaturation} (with $c$ in Lemma~\ref{HCL}), we perform an operation on $x$ as described below; the process stops when no such node exists.

    Now, we describe the operation on a node $x$. We take $C = C^x_0$ first. If there exists a \textit{rich} $k$-flat $F$ in $\mathbb{F}_q^n$, that means, $|C \cap F| \geq |C^x_0|/\sqrt{q}$, then we delete the subset $C \cap F$ from $C$. We repeat this deletion step until either (i) $|C| < |C^x_0|/2$ or (ii) no rich $k$-flat exists. We use $K_1,K_2,\dots,K_\ell$ to denote the deleted subsets. Note that every $K_i$ is a clique of $\mathcal{H}$. If the deletion steps ended with (i), then we create a child $y$ for $x$ such that $y$ copies the label of $x$, but with $C^x_0$ replaced by $C$ and all $K_i$'s appended.

    If the deletion steps ended with (ii), we consider an auxiliary hypergraph $\mathcal{H}'$ output by applying Lemma~\ref{counting_supersaturation} to $C$ as $P$ and $c$ as in Lemma~\ref{HCL}. Here, $|C| \geq |C^x_0|/2 \geq \theta q^{n-k}$ and the absence of rich flats are used. Thereafter, we apply Lemma~\ref{HCL} to $\mathcal{H}'$ with $\tau$ as given in Lemma~\ref{counting_supersaturation} to obtain a family $\mathfrak{C}$ of subsets of $C$. For each $C' \in \mathfrak{C}$, we create a child $y$ for $x$ with $C^y_0 = C'$, $C^y_1 = C^x_1, \dots, C^y_{\ell_x}=C^x_{\ell_x}$, $C^y_{\ell_x+1} = K_1,\dots, C^y_{\ell_x+\ell} = K_\ell$. Thus, the node operation is described. This process will give us a container-clique tree $T$ by the same argument as in Proposition~\ref{no-three-in-line_key}.

    We argue that each node operation shrinks the size of $C^x_0$ by a constant factor. This is self-evident if the deletion steps of $x$ ended with (i). If the deletion steps of $x$ ended with (ii), Lemma~\ref{HCL}(c) tells us that $|E(\mathcal{H}'[C'])| \leq (1-c)|E(\mathcal{H}')|$ for each $C' \in \mathfrak{C}$ in the node operation of $x$. It is also guaranteed that $\Delta_1(\mathcal{H}')\leq \theta |E(\mathcal{H}')|/|V(\mathcal{H}')|$. Thus, we can compute \begin{equation*}
        (|C|-|C'|) \cdot \theta\frac{|E(\mathcal{H'})|}{|V(\mathcal{H'})|}\geq(|C|-|C'|)\cdot\Delta_1(\mathcal{H}')\geq|E(\mathcal{H'}[C\setminus C'])|\geq c|E(\mathcal{H}')|.
    \end{equation*} Simplifying the inequality above gives $|C'| \leq (1-c/\theta)|C|$, which means $|C^y_0| \leq (1-c/\theta)|C^x_0|$ for every child $y$ of $x$.

    Next, we examine some quantitative properties of $T$. Firstly, since each node operation shrinks the size of $C^x_0$ by a constant factor, the height of $T$ is at most $O(\log q)$. Secondly, since each node operation extends the label by at most $\sqrt{q}$ additional cliques, every node has label length at most $O(\sqrt{q}\log q)$. By our choice of $\tau$ and Lemma~\ref{HCL}(b), the degree of a node will be at most $2^{O(q^{n-k-\epsilon}\log q)}$ with $\epsilon=1/(2r)$ as in Lemma~\ref{counting_supersaturation}, which implies that $T$ has at most $2^{O(q^{n-k-\epsilon}\log^2 q)}$ many leaves. Lastly, we have $|C^x_0| < 2\theta q^{n-k}$ for all leaf $x$.

    Finally, we upper bound the number of $(k,r)$-evasive sets in $\mathbb{F}_q^n$, which is exactly the number $\aleph$ of independent sets in $\mathcal{H}$. We denote the number of leaves as $\nu$, the maximum size of $C^x_0$ for leaf $x$ as $\chi$, the maximum size of $C^x_i$ for $i\neq 0$ as $\kappa$, and the maximum label length $\ell_x$ as $\lambda$. Since an independent set and a clique intersect in less than $r$ vertices, we have \begin{equation*}
        \aleph \leq \nu \cdot \binom{\kappa}{r}^\lambda \cdot 2^{\chi+r\lambda}  \leq 2^{O(q^{n-k-\epsilon}\log^2 q)} \cdot \binom{q^2}{r}^{\sqrt{q}\log q} \cdot 2^{O(q^{n-k})} \leq 2^{O(q^{n-k})}.
    \end{equation*} Therefore, we conclude the proof.
\end{proof}

\begin{proof}[Proof of Lemma~\ref{counting_supersaturation}]

    Let us write $|P| = m$. Note that $m \to \infty$ as $q \to \infty$ by our hypothesis. We use $\text{span}(S)$ to denote the minimal affine subspace containing $S \subset \mathbb{F}_q^n$, and consider the collection of $k$-sets in general position, that is, \begin{equation*}
        \mathcal{G} = \{K \subset P:~ |K| = k ~\text{and}~ \dim \text{span}(K) = k - 1\}.
    \end{equation*}

    We first solve the situation that $\mathcal{G} = o(m^k)$. We randomly choose $r$ points $v_1,\dots,v_r$ in $P$ with possible repetitions. For any subset $K$ obtained by taking $k$ entries from $(v_1,\dots,v_r)$, we have $\Pr[\dim \text{span}(K) = k - 1] < o(1)$ as a consequence of $\mathcal{G} = o(m^k)$. Let $S$ be the set $\{v_1,\dots,v_r\}$. Notice that if $\dim \text{span}(S) \geq k - 1$, then there exists $K \subset S$ with $|K| = k$ and $\dim \text{span}(K) = k - 1$. Hence we can use union bound to conclude  $\Pr[\dim \text{span}(S) \leq k] > 1 - o(1)$. This implies that there exists $\Omega(m^r)$ many $(k,r)$-sets in $P$. Thereafter, we take $E(\mathcal{H})$ to be all such $(k,r)$-sets. Because $\Delta_i(\mathcal{H}) \leq m^{r-i}$ trivially, we can check by computation that $\mathcal{H}$ satisfies the wanted property for some suitable $\theta$. 
    
    For the rest of this proof, we shall assume $\mathcal{G} = \Omega(m^k)$. We describe a random construction of $\mathcal{H}$. For each $K \in \mathcal{G}$, we notice $G_K = \text{span}(K)$ is a $(k-1)$-flat which is contained in $(1 \pm o(1)) q^{n-k}$ many $k$-flats, and the intersection of each pair of these $k$-flats is $G_K$. Since $|G_K \cap P| \leq 2m/\sqrt{q}$, we have $|P \setminus G_K| \geq m/2$. Let $\mathcal{F}_K$ be the collection of $k$-flats $F$ containing $G_K$ such that the quantity $m_{K,F} = |F \cap (P\setminus G_k)| \geq \frac{1}{3} \cdot m/q^{n-k}$. Note that $\sum_{F\in \mathcal{F}_K} m_{K,F} \geq m/3$ because flats not in $\mathcal{F}_K$ have small incidence contributions. Next, for each $(r-k)$-set $S$ in $F \cap (P\setminus G_k)$, we notice $K \cup S$ is a $(k,r)$-set, and we include it into $E(\mathcal{H})$ with probability\begin{equation*}
        p_{K,F} = \left( \frac{1}{m_{K,F}} \frac{m}{q^{n-k}} \right)^{r - k - 1}.
    \end{equation*} We remark that, although different choices of $K, F, S$ may result in the same $K \cup S$, each such $(k,r)$-set will only be considered at most $\binom{r}{k}$ many times. This is because $K \cup S$ spans hence also determines $F$.

    Next, we estimate the expectations of $|E(\mathcal{H})|$ and $\Delta_i(\mathcal{H})$. For the former, we can lower bound it easily \begin{align*}
        \mathbb{E}(|E(\mathcal{H})|) &\geq  \binom{r}{k}^{-1} \sum_{K \in \mathcal{G}} \sum_{F\in \mathcal{F}_K}\binom{m_{K,F}}{r-k} p_{K,F}\\
        &\gtrsim \sum_{K \in \mathcal{G}} \sum_{F\in \mathcal{F}_K} m_{K,F}^{r-k} \cdot \left( \frac{1}{m_{K,F}} \frac{m}{q^{n-k}} \right)^{r - k - 1}\\
        &\gtrsim \sum_{K \in \mathcal{G}} \sum_{F\in \mathcal{F}_K}m_{K,F} \cdot \left(\frac{m}{q^{n-k}} \right)^{r - k - 1}\\
        &\gtrsim \sum_{K \in \mathcal{G}} m \cdot \left(\frac{m}{q^{n-k}} \right)^{r - k - 1}\\
        &\gtrsim m^k \cdot m \cdot \left(\frac{m}{q^{n-k}} \right)^{r - k - 1}\\
        &= m^r/q^{(n-k)(r-k-1)}.
    \end{align*} In the computation above, we need to choose a sufficiently large $\theta$ depending on $k,n,r$. For the expectation of $\Delta_i(\mathcal{H})$, we fix distinct $v_1,\dots,v_i \in P$ arbitrarily and upper bound the expected number of edges in $\mathcal{H}$ containing them. To this end, we count the number of ways to choose $K,F,S$ such that $K \cup S$ contains $\{v_1,\dots,v_i\}$. Our counting depends on the number $|\{v_1, \dots, v_i\} \cap K|$ which we denote as $j$.
    
    In the case $j = i$, there are at most $m^{k - i}$ options for $K$ (and we denote them as $\mathcal{G}'$). Hence the expected number of edges from this case is at most \begin{equation*}
        \sum_{K \in \mathcal{G}'} \sum_{F \in \mathcal{F}_K} \binom{m_{K,F}}{r-k} p_{K,F} \lesssim \sum_{K \in \mathcal{G}'} m \cdot \left(\frac{m}{q^{n-k}} \right)^{r-k - 1} \lesssim m^{r-i}/q^{(n-k)(r-k-1)}.
    \end{equation*}

    In the case $0 \leq j < i$, there are at most $m^{k-j}$ options for $K$ (denoted as $\mathcal{G}'$), followed by only one option for $F$ (denoted as $F_K$) since a point in $\{v_1, \dots, v_i\} \setminus K$ will determine $F$, and $\binom{m_{K,F}}{r-k-i+j}$ options for $S$. So the expected number of edges from this case is at most
    \begin{equation*}
        \sum_{K \in \mathcal{G}'} \binom{m_{K,F_K}}{r-k-i+j} p_{K,F_K} \lesssim \sum_{K \in \mathcal{G}'} m_{K,F_K}^{j+1-i} \cdot \left(\frac{m}{q^{n-k}} \right)^{r-k - 1} 
        \lesssim m^{r-i}/q^{(n-k)(r-k + j -i)}.
    \end{equation*} Here, we used $j + 1 \leq i$ and $m_{K,F_K} = \Omega(m/q^{n-k})$ in the computation above.

    In the case $j = 0$ and $i > 1$, we can have a better estimate based on the following simple fact from linear algebra whose proof we skip.
    \begin{fact}
        For a set $S$ in an affine space of size at least two, if there are two points $u,v \not \in S$ such that $S \cup u$ is affinely independent, then there exists $w \in S$ such that $(S \setminus w) \cup \{u,v\}$ is affinely independent.
    \end{fact}

    
    According to this fact, in order to determine $K$ and $F$, it suffices to choose $k-1$ points from $P$ for $K$; and they will determine $F$ together with $v_1, v_2$ already in $S$; and then we can choose the last point for $K$ from $F \cap P$. After that, there are $\binom{m_{K,F}}{r-k-i}$ options for $S$, so the expected number of such edges, fixing $K$ and $F$, are \begin{equation*}
        \binom{m_{K,F}}{r-k-i} p_{K,F} \lesssim m_{K,F}^{1-i} \cdot \left(\frac{m}{q^{n-k}} \right)^{r-k - 1} \lesssim m^{r-k-i}/q^{(n-k)(r-k-i)}.
    \end{equation*} Here the inequality $m_{K,F} = \Omega(m/q^{n-k})$ is used. Together with the options for $K$ and $F$, the expected number of edges from this case is at most \begin{align*}
        m^{k-1} \cdot \left( 2m/\sqrt{q} \right) \cdot m^{r-k-i}/q^{(n-k)(r-k-i)} \lesssim m^{r-i}/q^{(n-k)(r-k-i)+1/2}.
    \end{align*} Here, crucially, we used the hypothesis that $|F \cap P| \leq 2m/\sqrt{q}$ for any $k$-flat $F$.

    Combining all three cases, we can check that, for all $2\leq i\leq r$, \begin{equation*}
        \mathbb{E}(\Delta_i(\mathcal{H})) \lesssim \frac{1}{\sqrt{q}} \left(\frac{q^{n-k}}{m}\right)^{i-1}\frac{\mathbb{E}(|E(\mathcal{H})|)}{m}, 
    \end{equation*} and $\mathbb{E}(\Delta_1(\mathcal{H})) \lesssim \mathbb{E}(|E(\mathcal{H})|)/m$. Thereafter, it is a standard application of Chernoff's bound to argue that such $\mathcal{H}$ with wanted properties exists deterministically. We will take $\theta$ to be sufficiently large to cover the hidden constant in the asymptotic notation $\lesssim$.
\end{proof}

\section{Remarks}\label{sec:remark}

A \textit{general position set} in $\mathbb{F}_q^n$ is a point set that is $(n-1,n+1)$-evasive. Considering all subsets of the moment curve $\{(x,x^2,\dots,x^n):~ x\in \mathbb{F}_q\}$, we know that $\mathbb{F}^n_q$ contains at least $2^q$ many general position sets. On the other hand, Bhowmick and Roche-Newton~\cite{bhowmick2022counting} proved that $\mathbb{F}^2_q$ contains at most $2^{q+q^{4/5+o(1)}}$ many general position sets. This result was later generalized by Chen--Liu--Nie--Zeng~\cite{chen2023random} to higher dimensions with the exponent $4/5$ replaced by $2n/(2n+1)$ for the number of general position sets in $\mathbb{F}^n_q$. With the help of container-clique trees, we can improve this upper bound to the following.
\begin{theorem}\label{general_position}
For fixed integer $n\ge 2$ and prime power $q\to\infty$, the number of general position sets in $\mathbb{F}^n_q$ is at most \begin{equation*}
    2^{q+q^{\frac{2}{3}+o(1)}}.
\end{equation*}
\end{theorem}

It is unclear how sharp this upper bound is, and its proof is similar to those of Theorems~\ref{counting} and~\ref{no-three-in-line}, hence we omit its proof but give a rough sketch instead. To prove Theorem~\ref{general_position}, we apply the container-clique tree process to the hypergraph $\mathcal{H}$, with $V(\mathcal{H}) = \mathbb{F}^n_q$ and $E(\mathcal{H})$ consisting of all $(n-1,n+1)$-sets. We should use a higher-dimensional version of Lemma~\ref{no-three-in-line_supersaturation} as the supersaturation result (Lemma~4.1 in \cite{chen2023random} suffices). For the node operation on $x$, a hyperplane $F$ will be determined as ``rich'' if \begin{equation*}
    |C\cap F| \gtrsim |C^x_0|/q^{2/3 - o(1)}.
\end{equation*} After the container-clique tree is constructed, the number of leaves and the maximum label length in the container-clique tree should both be upper bounded by $2^{q^{2/3+o(1)}}$, $|C^x_i| \leq q^n$ holds trivially for every node, and we should have $|C^x_0| \leq q+q^{2/3+o(1)}$ for every leaf.

\medskip \noindent {\bf Acknowledgement.} 
We wish to thank József Balogh, David Conlon, Alex Hof, Robert A. Krueger especially, and Jacques Verstraete for helpful comments and inspiring discussions.

\bibliographystyle{abbrv}
{\footnotesize\bibliography{main}}

\end{document}